\documentclass[11pt]{amsart}
\usepackage{hyperref, amssymb, fullpage, environ}
\usepackage{tikz-cd}
\usepackage[capitalise, nameinlink]{cleveref}
\usepackage[all]{xy}

\makeatletter
\@namedef{subjclassname@2010}{%
  \textup{2010} Mathematics Subject Classification}
\makeatother
\newtheorem{mainthm}{Theorem}

\newtheorem{theorem}{Theorem}[section]
\newtheorem{prop}[theorem]{Proposition}
\newtheorem{lem}[theorem]{Lemma}

\theoremstyle{definition}
\newtheorem{defn}[theorem]{Definition}
\newtheorem{rem}[theorem]{Remark}
\newtheorem{ex}[theorem]{Example}

\def\R{{\mathbb{R}}}
\def\N{{\mathbb{N}}}
\def\C{{\mathbb{C}}}
\def\Q{{\mathbb{Q}}}

\newcommand{\T}{\mathbb{T}}
\renewcommand{\t}{\mathcal{T}}
\renewcommand{\b}{\mathcal{B}}

\newcommand{\U}{\widehat{\mathcal{U}}}
\renewcommand{\u}{\mathcal{U}}

\DeclareMathOperator{\tr}{Tr}
\DeclareMathOperator{\diag}{diag}

\NewEnviron{eqsplit}{
\begin{equation*}
\begin{split}
  \BODY
\end{split}
\end{equation*}
}

\hypersetup{
	colorlinks,
	linkcolor=blue,          
	filecolor=magenta,      
	urlcolor=cyan           
}

\begin{document}
\title[$K$-Stability of A$\mathbb{T}$-algebras]{$K$-Stability of A$\mathbb{T}$-algebras}
\author[Apurva Seth, Prahlad Vaidyanathan]{Apurva Seth, Prahlad Vaidyanathan}
\address{Department of Mathematics\\ Indian Institute of Science Education and Research Bhopal\\ Bhopal ByPass Road, Bhauri, Bhopal 462066\\ Madhya Pradesh. India.}
\email{apurva17@iiserb.ac.in, prahlad@iiserb.ac.in}
\date{}
\subjclass[2010]{Primary 46L85; Secondary 46L80}
\keywords{Nonstable K-theory, C*-algebras}
\maketitle
\parindent 0pt

\begin{abstract}
We describe a procedure to compute the rational nonstable K-groups of A$\T$-algebras. As an application, we show that an A$\T$-algebra is K-stable if and only if it has slow dimension growth.
\end{abstract}

\section{Introduction}
Given a unital C*-algebra $A$, let $\mathcal{U}_n(A)$ denote the group of $n\times n$ unitary matrices over $A$. This is a topological group, and its homotopy groups $\pi_k(\mathcal{U}_n(A))$ are collectively referred to as the \emph{nonstable $K$-theory} groups of $A$. The study of these groups goes back to the work of Kuiper \cite{kuiper}, who proved that $\u(\b(H))$ is contractible whenever $H$ is an infinite dimensional Hilbert space. Since then, these groups were calculated for a variety of individual C*-algebras (over many years). These ideas were eventually placed in the broader context of noncommutative topology by Thomsen \cite{thomsen}. In particular, he introduced the notion of a quasi-unitary, and was thus able to study non-unital and unital C*-algebras on the same footing. \\

If $\U_n(A)$ denotes the group of quasi-unitaries in $M_n(A)$, one is then faced with the problem of computing $\pi_k(\U_n(A))$. Unfortunately, for an arbitrary C*-algebra $A$, such computations are prohibitively difficult. In fact, these groups are largely unknown even for the algebra of complex numbers! In order to alleviate this difficulty, we studied these groups \emph{upto rationalization} in \cite{apurva_pv_af}. Here, tools from rational homotopy theory allowed us to compute $\pi_k(\U_n(A))\otimes \Q$ when $A$ is an AF-algebra. Indeed, maps between finite dimensional C*-algebras translated to maps between $\Q$-vector spaces in a simple and predictable manner. \\

In this paper, we look to understand A$\T$-algebras along similar lines. Our first result is an explicit calculation of the groups $\pi_k(\U_n(A))\otimes \Q$ when $A$ is an A$\T$-algebra (\cref{thm_at_alg_rat_hom}). This calculation is a consequence of the fact that maps between two circle algebras may be modified (upto homotopy) into maps of a particularly nice form. \\

We then use these results to obtain a characterization of A$\T$-algebras that have slow dimension growth (see \cref{defn_slow_dim_growth}). To put it this in context, observe that the groups $\pi_k(\U_n(A))$ depend, in general, on the matrix size $n$. However, for certain classes of C*-algebras, $\pi_k(\U(A))$ is naturally isomorphic to $\pi_k(\U_n(A))$ for all $n\geq 1$ and all $k\geq 0$. In particular, for such an algebra,
\[
\pi_k(\U(A)) \cong \begin{cases}
K_0(A) &: \text{ if } k \text{ is odd, and} \\
K_1(A) &: \text{ if } k \text{ is even}.
\end{cases}
\]
Thomsen \cite{thomsen} had proved that the Cuntz algebras and simple, infinite dimensional AF-algebras have this property, and he termed this phenomenon \emph{$K$-stability}. A variety of interesting C*-algebras are now known to have this property (see \cite[Remark 1.5]{apurva_pv_cx}). Similarly, we say that a C*-algebra $A$ is \emph{rationally} $K$-stable if the groups $\pi_k(\U_n(A))\otimes \Q$ are all naturally isomorphic to one another. We had proved in \cite{apurva_pv_af} that these two notions coincide for AF-algebras. Our main result here is that the same is true for all A$\T$-algebras, and that it is equivalent to slow dimension growth.

\begin{mainthm}\label{mainthm_k_stable}
For an A$\T$-algebra $A$, the following are equivalent:
\begin{enumerate}
\item $A$ is $K$-stable.
\item $A$ is rationally $K$-stable.
\item $A$ has slow dimension growth.
\end{enumerate}
\end{mainthm}

The paper is organized as follows. In \cref{pre}, we briefly recall the preliminary results and definitions that will be used throughout the paper. In \cref{sec_hom_circle}, we describe a simple class of $\ast$-homomorphisms between circle algebras, and show that, for our purposes, understanding such maps is sufficient. This reduction in turn results in \cref{thm_at_alg_rat_hom}, which is the main goal of \cref{sec_rat_hom_at}. Finally, \cref{sec_k_stability_at} is devoted to a proof of \cref{mainthm_k_stable}.

\section{Preliminaries}\label{pre}
We begin by reviewing the work of Thomsen of constructing the nonstable K-groups associated to a C*-algebra. For the proofs of all the facts mentioned below, the reader may refer to \cite{thomsen}. \\

Let $A$ be a C*-algebra (not necessarily unital). Define an associative composition $\cdot$ on $A$ by
\[
a\cdot b=a+b-ab
\]
An element $u\in A$ is said to be a quasi-unitary if $u\cdot u^{\ast} = u^{\ast}\cdot u = 0$, and we write $\U(A)$ for the set of all quasi-unitary elements in $A$. Moreover, we write $\U_n(A)$ for the group $\U(M_n(A))$.\\


\begin{defn}
Let $A$ be a C*-algebra, and $k\geq 0$ and $m\geq 1$ be integers. Define
\[
G_k(A) :=\pi_k(\U(A)), \text{ and } F_m(A) := \pi_m(\U(A))\otimes \Q.
\]
\end{defn}

Note that if $A$ is unital and $\u(A)$ denotes the group of unitaries in $A$, then the map $\U(A) \to \u(A)$ given by $u \mapsto (1-u)$ induces a natural isomorphism of groups. In particular,
\[
G_k(A) \cong \pi_k(\u(A)) \text{ and } F_m(A) \cong \pi_m(\u(A))\otimes \Q
\]

Recall \cite{schochet} that a homology theory on the category of $C^{\ast}$-algebras is a sequence $(h_n)$ of covariant, homotopy invariant functors from the category of C*-algebras to the category of abelian groups such that, if $0 \to J\xrightarrow{\iota} B\xrightarrow{p} A\to 0$ is a short exact sequence of C*-algebras, then for each $n \in \N$, there exists a connecting map $\partial : h_n(A)\to h_{n-1}(J)$, making the following sequence exact
\[
\ldots \xrightarrow{\partial} h_n(J)\xrightarrow{h_n(\iota)} h_n(B) \xrightarrow{h_n(p)} h_n(A)\xrightarrow{\partial} h_{n-1}(J)\to \ldots
\]
and furthermore, $\partial$ is natural with respect to morphisms of short exact sequences. Finally, we say that a homology theory $(h_n)$ is continuous if, whenever $A = \lim A_i$ is an inductive limit in the category of C*-algebras, then $h_n(A) = \lim h_n(A_i)$ in the category of abelian groups. The next proposition is a consequence of \cite[Proposition 2.1]{thomsen} and \cite[Theorem 4.4]{handelman}.

\begin{prop}\label{prop: continuous_homology}
Both $(G_k)$ and $(F_m)$ are continuous homology theories.
\end{prop}

Let us now describe the behaviour of the functor $F_m$ on finite dimensional C*-algebras. From now on, we write $M_n$ for the algebra $M_n(\C)$, and $\N_0$ for the set of all non-negative integers. Now suppose $A = \bigoplus_{j=1}^K M_{n_j}$ and $B = \bigoplus_{i=1}^L M_{\ell_i}$ are two finite dimensional C*-algebras, and $\varphi:A\to B$ is a $\ast$-homomorphism. Then, $\varphi$ is determined upto unitary equivalence by an $L\times K$ matrix $\Phi \in M_{L\times K}(\N_0)$, called the multiplicity matrix of $\varphi$ (see, for instance, \cite[Section III.2]{davidson}). One of the main results of \cite{apurva_pv_af} was the following: For each $m\geq 1$,
\[
F_m(A) = \bigoplus_{j=1}^K \Q^{d(m,j)} \text{ where } d(m,j) = \begin{cases}
1 &: \text{ if } 1\leq m\leq 2n_j-1, m\text{ odd}, \\
0 &: \text{ otherwise}
\end{cases}
\]
and $F_m(B)$ has an analogous expression. Moreover, $F_m(\varphi):F_m(A)\to F_m(B)$ is represented as multiplication by the matrix $\Phi$. We now use this as the starting point to study A$\T$-algebras.
%
%

\section{$\ast$-Homomorphisms between Circle Algebras}\label{sec_hom_circle}

Let $\T := \{z \in \C : |z| = 1\}$ denote the unit circle. A circle algebra is an algebra of the form $C(\T)\otimes F$ where $F$ is a finite dimensional C*-algebra. In this section, we will describe a particularly tractable class of $\ast$-homomorphisms between two circle algebras, and prove that they are generic in a certain sense. We begin by revisiting a result of Thomsen \cite{thomsen_circle}, where he classified such maps upto approximate unitary equivalence.  \\

Let $\t^n$ denote the set of unordered $n$-tuples in $\T$ (with $\t^0=\emptyset$).  If $\Sigma_n$ denotes the symmetric group on $n$ letters, we may think of $\t^n$ as a closed subset of $\C^n/\Sigma_n$. Let
\[
A = \bigoplus_{j=1}^K C(\T)\otimes M_{n_j}, \text{ and } B = \bigoplus_{i=1}^L C(\T)\otimes M_{\ell_i}.
\]
Let $\varphi: A\to B$ be a $\ast$-homomorphism and let $u$ denote the canonical unitary generator of $C(\T)$. Fix $1\leq j\leq K$, and choose minimal projections $e_j\in M_{n_j}$. Set $u_j=(0,0,\ldots,0,u\otimes e_j,0,\ldots,0)\in A$ and write 
\[
\varphi(u_j)=(w_{1,j}, w_{2,j},\hdots, w_{L,j}) \in B
\]
For each $1\leq i\leq L$, $p_{i,j}=w_{i,j}w_{i,j}^{\ast}=w_{i,j}^{\ast}w_{i,j}$ is a projection in $C(\T)\otimes M_{\ell_i}$, so the value of $\tr(w_{ij}w_{ij}^{\ast}(z))$ does not vary with $z\in\T$. Let $a_{i,j} \in \N\cup\{0\}$ be that constant value. Define $\widehat{\varphi}_{i,j}:\T\to \t^{a_{i,j}}$ by 
\[
\widehat{\varphi}_{i,j}(z)=[\lambda\in \sigma(w_{i,j}(z)): |\lambda|=1].
\]
It is a fact that the maps $\widehat{\varphi}_{i,j}$ are all continuous. The values $\{a_{i,j}\}$ are called the \emph{multiplicity constants} of $\varphi$, and $\{\widehat{\varphi}_{i,j}\}$ are called the \emph{characteristic functions} for $\varphi$. Thomsen's main result is that the $\ast$-homomorphism $\varphi$ is uniquely determined (upto approximate unitary equivalence) by these quantities.

\begin{theorem}[\cite{thomsen_circle}, Theorem 2.1]\label{thomsen_1}
~\begin{enumerate}
\item Two $\ast$-homomorphisms $\varphi,\psi:A\to B$ are approximately unitarily equivalent if and only if $\widehat{\varphi}_{i,j}=\widehat{\psi}_{i,j}$ for all $i=1,2,\ldots, L$ and $j=1,2,\ldots, K$.
\item Let $\{a_{i,j}: 1\leq i\leq L, 1\leq j\leq K\}$ be a set of non-negative integers such that $\sum_{j=1}^{K}a_{ij}n_j\leq m_i$ for each $i \in \{1,2,\ldots, L\}$, and let $\eta_{i,j}:\T\to \t^{a_{i,j}}$ be continuous maps. Then, there exists a $\ast$-homomorphism $\varphi:A\to B$ such that $\widehat{\varphi}_{i,j}=\eta_{i,j}$ for all $i,j$.
\end{enumerate}
\end{theorem}
\begin{proof}
Since the construction is crucial for our needs, we recall the proof of part (2). For clarity, we break it into cases. In what follows we will identify $C(\T)$ with $\{f \in C[0,1] : f(0) = f(1)\}$.\\

Suppose first that $A = C(\T)\otimes M_n$ and $B = C(\T)\otimes M_{\ell}$ so that $na \leq \ell$, and $\eta : \T\to \t^a$ is a continuous function. By \cite[Theorem 5.2 of Chapter II]{kato}, there are continuous functions $\lambda_1, \lambda_2, \ldots, \lambda_a : [0,1]\to \T$ such that
\[
\eta(e^{2\pi it}) = [\lambda_1(t), \lambda_2(t), \ldots, \lambda_a(t)]
\]
for all $t\in [0,1]$, and there is a permutation $\sigma \in \Sigma_a$ such that $\lambda_{\sigma(p)}(0)=\lambda_{p}(1)$ for each $p \in \{1,2,\ldots, a\}$. Let $v_{\sigma} \in \u_a$ be the permutation matrix obtained from the identity matrix by permutating its rows according to $\sigma$. Let $w:[0,1]\to \u_a$ be a continuous path with $w(0) = 1$ and $w(1) = v_{\sigma}$, and let $u : [0,1]\to \u_{\ell}$ be given by $u(t) = \diag(w(t)\otimes I_n, I_{\ell-na})$. We may then define $\varphi : C(\T)\otimes M_n\to C(\T)\otimes M_{\ell}$ by
\begin{equation}\label{generic_hom}
\varphi(f)(t)=u(t)\begin{pmatrix}
f(\lambda_1(t))& 0 &\ldots & 0 & 0\\
0 & f(\lambda_2(t))& \ldots & 0 & 0 \\
\vdots &\vdots &\vdots & \vdots & \vdots \\
0& 0 & \ldots & f(\lambda_a(t)) & 0 \\
0 & 0 & \ldots & 0 & 0_{m-na}
\end{pmatrix} u(t)^{\ast}.
\end{equation}
Then, $\varphi$ is a well-defined $\ast$-homomorphism with characteristic function $\eta$.\\

Now suppose $A = \bigoplus_{j=1}^K C(\T)\otimes M_{n_j}$ and $B = C(\T)\otimes M_{\ell}$. Then, $\{a_1, a_2, \ldots, a_K\}$ are non-negative integers such that $\sum_{j=1}^K a_jn_j \leq \ell$ and $\eta_j:\T\to \t^{a_j}$ are continuous functions. For $j \in \{1,2,\ldots, K\}$, define a $\ast$-homomorphism $\varphi_j : C(\T)\otimes M_{n_j}\to C(\T)\otimes M_{a_{1,j}n_j}$ whose characteristic function is $\eta_{j}$ using the recipe from \cref{generic_hom}. We may then define $\varphi : A\to B$ by
\[
(g_1,g_2,\hdots,g_K)\mapsto \diag(\varphi_1(g_1),\varphi_2(g_2),\hdots, \varphi_K(g_K),0,0\hdots,0).
\]
Then, $\varphi$ is a $\ast$-homomorphism with multiplicity constants $\{a_1,a_2,\ldots, a_K\}$ and characteristic functions $\{\eta_{1}, \eta_{2},\ldots, \eta_{K}\}$. \\

Finally, if $A = \bigoplus_{j=1}^K C(\T)\otimes M_{n_j}$ and $B = \bigoplus_{i=1}^L C(\T)\otimes M_{\ell_i}$, then we may define $\varphi_i : A \to C(\T)\otimes M_{\ell_i}$ using the above recipe with multiplicity constants $\{a_{i,1}, a_{i,2}, \ldots, a_{i,K}\}$ and characteristic functions $\{\eta_{i,1}, \eta_{i,2}, \ldots, \eta_{i,K}\}$. Then define $\varphi : A\to B$ by $\varphi(a) = (\varphi_1(a), \varphi_2(a), \ldots, \varphi_L(a))$. Once again, this map satisfies the required properties.
\end{proof}

The proof of this theorem allows us to make the following definition.

\begin{defn}\label{def:gen_hom}
A $\ast$-homomorphism between two circle algebras is said to be of \emph{Type A} if it is in the form described in the preceding proof. 
\end{defn}

With this terminology in place, Thomsen's theorem merely states that every $\ast$-homomorphism between circle algebras is approximately unitarily equivalent to a $\ast$-homomorphism of Type A. \\

Now consider the special case of the earlier proof with $A = C(\T)\otimes M_n$ and $B = C(\T)\otimes M_{\ell}$. Then, a $\ast$-homomorphism $\varphi : A\to B$ of Type A is described by the following data:
\begin{itemize}
\item A non-negative integer $a$ with $an\leq \ell$ (which is the multiplicity of $\varphi$).
\item A permutation $\sigma_{\varphi}\in \Sigma_a$.
\item Continuous maps $\lambda_{p} : [0,1]\to \T$ such that $\lambda_{\sigma_{\varphi}(p)}(0) = \lambda_{p}(1)$ for each $p \in \{1,2,\ldots, a\}$.
\item A continuous path $w : [0,1]\to \u_a$ with $w(0) = I_a$ and $w(1) = v_{\sigma_{\varphi}}$ (where $v_{\sigma_{\varphi}}$ is the permutation matrix associated to $\sigma_{\varphi})$.
\end{itemize}
The tuple $(a, \sigma_{\varphi}, \lambda_1, \lambda_2, \ldots, \lambda_a, w)$ will henceforth referred to as the \emph{data tuple} associated to $\varphi$. For convenience, we will write $\sigma_0$ for the identity permutation, and $w_0$ for the constant path at the identity. A $\ast$-homomorphism $\varphi : A\to B$ is said to be of 
\begin{itemize}
\item Type B if it is of Type A and $\sigma_{\varphi} = \sigma_0$.
\item Type C if it is of Type B and $w = w_0$.
\item Type D if it is of Type C and each $\lambda_{p}$ is a loop in $\T$ based at $1$. 
\end{itemize}

Now suppose $A = \bigoplus_{j=1}^K C(\T)\otimes M_{n_j}$ and $B = \bigoplus_{i=1}^L C(\T)\otimes M_{\ell_i}$. Then, a $\ast$-homomorphism $\varphi : A\to B$ of Type A can be built up from $\ast$-homomorphisms $\varphi_{i,j} : C(\T)\otimes M_{n_j} \to C(\T)\otimes M_{\ell_i}$ as in the proof above. We say that $\varphi$ is of Type B (respectively of Type C or Type D) if each $\varphi_{i,j}$ of Type B (respectively of Type C or Type D). \\

We now wish to show that every $\ast$-homomorphism of Type A is homotopic to a $\ast$-homomorphism of Type B. Before we prove this, we introduce some notations that will be used in the proof. Let $\lambda_1,\lambda_2:[0,1]\to\T$ be two continuous maps such that $\lambda_1(1)=\lambda_2(0)$. Then, for $s\in[0,1]$, define $\lambda_i^s,\lambda_{1,2}^s:[0,1]\to\T$ by
\[
\lambda_{1,2}^s(t)= \left\{
\begin{aligned}
& \lambda_1(2st + 1 -s) &: \text{ if } 0\leq t\leq \frac{1}{2} \\
& \lambda_2(2t-1) &: \text{ if } \frac{1}{2}\leq t\leq 1\\
\end{aligned}
\right\}
\text{ and }\lambda_i^s(t)=\lambda_i(t(1-s)).
\]
We also write $\lambda_{1,2} := \lambda_{1,2}^1$.

\begin{theorem}\label{thm_homotopy_type_ab}
A $\ast$-homomorphism of Type A between two circle algebras is homotopic to a $\ast$-homomorphism of Type B with the same multiplicity constants.
\end{theorem}
\begin{proof}
Let $A = \bigoplus_{j=1}^K C(\T)\otimes M_{n_j}$ and $B = \bigoplus_{i=1}^L C(\T)\otimes M_{\ell_i}$, and let $\varphi : A\to B$ be a $\ast$-homomorphism of Type A. Then, the components of $\varphi$, given by $\varphi_i : A\to C(\T)\otimes M_{\ell_i}, i=1,2,\ldots, L$, are themselves $\ast$-homomorphisms of Type A. Furthermore, each $\varphi_i$ is of the form
\[
\varphi_i(g_1,g_2,\ldots, g_K) = \diag\Big(\varphi_{i,1}(g_1), \varphi_{i,2}(g_2), \ldots, \varphi_{i,K}(g_K), 0, \ldots, 0 \Big).
\]
where each map $\varphi_{i,j} : C(\T)\otimes M_{n_j} \to C(\T)\otimes M_{\ell_i}$ is of Type A. Hence, to prove the theorem, it suffices to assume that $A = C(\T)\otimes M_n, B = C(\T)\otimes M_{na}$ and $\varphi : A\to B$ is a unital $\ast$-homomorphism of Type $A$. Let $(a, \sigma_{\varphi}, \lambda_1, \lambda_2, \ldots, \lambda_a, w)$ be the data tuple associated to $\varphi$ so that
\[
\varphi(f)(t)=u(t)\begin{pmatrix}
f(\lambda_1(t))& 0 &\hdots & 0\\
0 & f(\lambda_2(t))& \hdots & 0\\
\hdots &\hdots &\hdots & \hdots\\
0& 0 & \hdots & f(\lambda_a(t))
\end{pmatrix} u(t)^{\ast}.
\]
where $u:[0,1]\to \u_{na}$ is the path $u(t) = w(t)\otimes I_n$. We now break the proof into cases for convenience.
\begin{enumerate}
\item Suppose $\sigma_{\varphi}$ is a transposition in $\Sigma_a$, then for simplicity of notation we assume $\sigma_{\varphi} = (1,2)$. Then $\lambda_1(0) = \lambda_2(1), \lambda_1(1) = \lambda_2(0)$, and the other $\lambda_{p}$ are all loops. Consider $H:[0,1]\times A\to B$ given by
\[
H(s,f)(t) = u(t)\diag\Big(f(\lambda_1^s(t)), f(\lambda_{1,2}^s(t)), f(\lambda_3(t)), \ldots, f(\lambda_a(t)) \Big)u(t)^{\ast}.
\]
Then, $H$ is well-defined because for any $f\in A$ and $s\in [0,1]$,
\begin{eqsplit}
H(s,f)(1) &= v_{\sigma_{\varphi}}\diag\Big(f(\lambda_1(1-s)), f(\lambda_2(1)), f(\lambda_3(1)), \ldots, f(\lambda_a(1)) \Big)v_{\sigma_{\varphi}}^{\ast} \\
&= \diag\Big(f(\lambda_2(1)), f(\lambda_1(1-s)), f(\lambda_3(1)), \ldots, f(\lambda_a(1)) \Big) \\
&= \diag\Big(f(\lambda_1(0)), f(\lambda_1(1-s)), f(\lambda_3(0)), \ldots, f(\lambda_a(0)) \Big) \\
&= H(s,f)(0).
\end{eqsplit}
Furthermore, $\lambda_{1,2}^0 \sim_h \lambda_2$, so
\[
H(0,f)(t) = u(t)\diag\Big(f(\lambda_1(t)), f(\lambda_{1,2}^0(t)), f(\lambda_3(t)), \ldots, f(\lambda_a(t)) \Big)u(t)^{\ast}
\]
defines a $\ast$-homomorphism such that $H(0,\cdot) \sim_h \varphi$. Also,
\[
H(1,f)(t) = u(t)\diag\Big(f(\lambda_1(0)), f(\lambda_{1,2}(t)), f(\lambda_3(t)), \ldots, f(\lambda_a(t)) \Big)u(t)^{\ast}
\]
which is of Type B since $\lambda_{1,2}$ is a loop.
\item Now suppose $\sigma_{\varphi}$ is a cycle of length $n$, and assume by induction that the result is true for any $\ast$-homomorphism of Type A whose associated permutation is a cycle of length $<n$. Once again, for simplicity, we assume $\sigma_{\varphi} = (1,2,\ldots, n)$. Then, $\lambda_1(1) = \lambda_2(0)$, $\lambda_2(1) = \lambda_3(0)$, $\ldots, \lambda_{n-1}(1) = \lambda_n(0)$, $\lambda_n(1) = \lambda_1(0)$. Also, the other $\lambda_{p}$ are all loops. Define $H:[0,1]\times A \to B$ by
\[
H(s,f)(t) = u(t)\diag\Big(f(\lambda_1^s(t)), f(\lambda_{1,2}^s(t)), f(\lambda_3(t)), \ldots, f(\lambda_n(t)), f(\lambda_{n+1}(t)), \ldots, f(\lambda_a(t)) \Big)u(t)^{\ast}.
\]
Then, $H$ is well-defined because for any $f\in A$ and $s\in [0,1]$,
\begin{eqsplit}
H(s,f)(1) &= v_{\sigma_{\varphi}}\diag\Big(f(\lambda_1(1-s)), f(\lambda_2(1)), f(\lambda_3(1)), \ldots, f(\lambda_n(1)), f(\lambda_{n+1}(1)), \ldots, f(\lambda_a(1)) \Big)v_{\sigma_{\varphi}}^{\ast} \\
&= \diag\Big(f(\lambda_n(1)), f(\lambda_1(1-s)), f(\lambda_2(1)), \ldots, f(\lambda_{n-1}(1)), f(\lambda_{n+1}(1)), \ldots, f(\lambda_a(1)) \Big) \\
&= \diag\Big(f(\lambda_1(0)), f(\lambda_1(1-s)), f(\lambda_3(0)), \ldots, f(\lambda_n(0)), f(\lambda_{n+1}(0)), \ldots, f(\lambda_a(0)) \Big) \\
&= H(s,f)(0).
\end{eqsplit}
As before, since $\lambda_{1,2}^0 \sim_h \lambda_2$, it follows that $H(0,\cdot) \sim_h \varphi$. Now note that
\[
H(1,f)(t) = u(t)\diag\Big(f(\lambda_1(0)), f(\lambda_{1,2}(t)), f(\lambda_3(t)), \ldots, f(\lambda_n(t)), f(\lambda_{n+1}(t)), \ldots, f(\lambda_a(t)) \Big)u(t)^{\ast}.
\]
Thus, $H(1,\cdot) : A\to B$ is a $\ast$-homomorphism of Type A whose associated permutation $\tau$ is the cycle $\tau = (2,3,\ldots, n)$. By induction, $H(1,\cdot)$ is homotopic to a $\ast$-homomorphism $\psi : A\to B$ of Type B, and thus $\varphi \sim_h \psi$.
\item Now suppose $\sigma_{\varphi} \in \Sigma_a$ is any permutation, then we may express $\sigma_{\varphi}$ as a product of disjoint cycles $\sigma_{\varphi} = \sigma_1\sigma_2\ldots \sigma_k$. Then each $\sigma_i$ may be reduced to the identity permutation (keeping the multiplicity intact) by part (2). This reduces $\sigma_{\varphi}$ to the identity permutation, proving the result.
\end{enumerate}
\end{proof}

Let $\varphi : C(\T)\otimes M_n\to C(\T)\otimes M_{\ell}$ be a $\ast$-homomorphism of Type B, and let $(a, \sigma_0, \lambda_1, \lambda_2, \ldots, \lambda_a, w)$ be a data tuple associated to $\varphi$. Then, by construction $w:[0,1]\to \u_a$ is a path based at the identity, and hence defines a unitary $w \in C(\T)\otimes M_a$. Therefore, $\varphi$ takes the form
\[
\varphi(f)(t) = u(t)\diag\Big(f(\lambda_1(t)), f(\lambda_2(t)), \ldots, f(\lambda_a(t))u(t)^{\ast}
\]
where $u$ is now a unitary in $C(\T)\otimes M_m$ given by $u(t) = \diag(w(t)\otimes I_n,I_{\ell-na})$. We conclude that

\begin{prop}\label{prop_type_bc}
A $\ast$-homomorphism between two circle algebras of Type B is unitarily equivalent to one of Type C.
\end{prop}

Now let $\varphi : C(\T)\otimes M_n\to C(\T)\otimes M_{\ell}$ be a $\ast$-homomorphism of Type C, and let $(a, \sigma_0, \lambda_1, \lambda_2, \ldots, \lambda_a, w_0)$ be a data tuple associated to $\varphi$. Then, by construction, each $\lambda_{p}$ is a loop in $\T$. If $b_{p}$ denotes the winding number of the loop $\lambda_{p}$, then there exists $c_{p} \in \T$ such that
\[
\lambda_{p} \sim_h c_{p}\delta_{b_{p}}
\]
where $\delta_n : \T\to \T$ denotes the map $z \mapsto z^n$. Hence, $\varphi$ is homotopic to a $\ast$-homomorphism of Type C whose data set takes the form $(a, \sigma_0, c_1\delta_{b_1}, c_2\delta_{b_2}, \ldots, c_a\delta_{b_a}, w_0)$. Now fix $1\leq p\leq a$, and choose $d_{p} \in \R$ such that $c_{p} = e^{2\pi i d_{p}}$. Then, the function $F_{p}:[0,1]\times C(\T) \to C(\T)$ given by
\[
F_{p}(s,f)(z) = f(e^{2\pi i sd_{p}}z^{b_{p}})
\]
is a homotopy from $c_{p}\delta_{b_{p}}$ to $\delta_{b_{p}}$. Applying these homotopies along the diagonal as before, we see that $\varphi$ is homotopic to a $\ast$-homomorphism of Type C whose data set is $(a, \sigma_0, \delta_{b_1}, \delta_{b_2}, \ldots, \delta_{b_a}, w_0)$. Applying this componentwise to an arbitrary $\ast$-homomorphism of Type C, we conclude that

\begin{prop}\label{prop_type_cd}
A $\ast$-homomorphism between two circle algebras of Type C is homotopic to one of Type D.
\end{prop}

Since we are interesting in continuous homology theories applied to inductive limits of circle algebras, it suffices for our purposes to consider homomorphisms of Type D. Therefore, we will henceforth refer to $\ast$-homomorphisms of Type D as \emph{diagonal} maps. 

\section{Calculation of $F_m(A)$ when $A$ is an A$\T$-algebra}\label{sec_rat_hom_at}

Recall that, for a C*-algebra $A$ and for $m\geq 1$, $F_m(A) =\pi_m(\U(A))\otimes \Q$. The aim of this section is to compute $F_m(\varphi)$ when $\varphi$ is a diagonal $\ast$-homomorphism between circle algebras, and then use it to compute $F_m(A)$ when $A$ is an A$\T$-algebra. 

\begin{prop}\label{prop_iso_fm}
If $A= \bigoplus_{j=1}^K C(\T)\otimes M_{n_j}$, then for $m\geq 1$, $F_m(A)$ is given by
\begin{eqsplit}
F_m(A) &\cong \bigoplus_{j=1}^{K} \pi_m(\mathcal{U}_{n_j})\otimes\Q\oplus\pi_{m+1}(\mathcal{U}_{n_j})\otimes\Q \\
&\cong \bigoplus_{j=1}^K \Q^{d(m,j)} \text{ where } d(m,j) = \begin{cases}
1 &: \text{ if } 1\leq m \leq 2n_j-1, \\
0 &: \text{ otherwise}.
\end{cases}
\end{eqsplit}
\end{prop} 
\begin{proof}
If $D = \bigoplus_{j=1}^K M_{n_j}$, then there is a split exact sequence $0 \to C_{\ast}(\T, D) \to A \to D \to 0$, where $C_{\ast}(\T, D) = \{f \in C(\T,D) : f(1) = 0\} \cong SD$, the suspension of $D$. Since $F_m$ is a homology theory, it follows that $F_m(A) \cong F_m(D)\oplus F_m(SD) \cong F_m(D)\oplus F_{m+1}(D)$. The result now follows from additivity of $F_m(\cdot)$ and the fact that
\[
F_m(M_n) = \pi_m(\u_n)\otimes \Q \cong \begin{cases}
\Q &: \text{ if } 1 \leq m\leq 2n-1, m \text{ odd} \\
0 &: \text{ otherwise}.
\end{cases}
\]
by \cite[Example 1.6]{apurva_pv_af}.
\end{proof}

Since it appears frequently below, we define
\[
V_n^m := \pi_m(\u_n)\otimes \Q \oplus \pi_{m+1}(\u_n)\otimes \Q.
\]
For convenience of notation, we will often treat entries in $V_n^m$ as pairs $(x,y)$ with $x \in \pi_m(\u_n)\otimes \Q$ and $y\in \pi_{m+1}(\u_n)\otimes \Q$ (even though both terms cannot be non-zero simultaneously).

\begin{prop}\label{prop_fm_circle}
For $1\leq p \leq a$, let $\lambda_{p} : \T\to \T$ be continuous loops based at $1$ with winding number $w(\lambda_{p})$. Let $\varphi:C(\T)\otimes M_n\to C(\T)\otimes M_{\ell}$ be the diagonal $\ast$-homomorphism given by
\[
f\mapsto\diag(f\circ\lambda_1,f\circ\lambda_2,\hdots,f\circ\lambda_a, 0_{\ell-na}).
\]
Then for any $m\geq 1$, $F_m(\varphi): V_n^m\to V_{\ell}^m$ is given by
\[
(x,y) \mapsto \left(ax, by\right)
\]
where $b = \sum_{p=1}^a w(\lambda_{p})$.
\end{prop}
The pair $(a,b)$ will henceforth be referred to as the \emph{signature} of the diagonal map $\varphi$.
\begin{proof}
Let $\iota : C(\T)\otimes M_{na}\to C(\T)\otimes M_{\ell}$ be the inclusion map $x \mapsto \diag(x,0)$. By the naturality of the isomorphism in \cref{prop_iso_fm} and \cite[Lemma 2.2]{apurva_pv_af},
\[
F_m(\iota) = \begin{cases}
\text{id} &: \text{ if } 1 \leq m\leq 2na-1, \\
0 &: \text{ otherwise}.
\end{cases}
\]
Therefore, it suffices to consider the case when $\ell=na$. Furthermore, we may assume without loss of generality that $a=2$. In that case, for $p \in \{1,2\}$, let $\varphi_{p} : C(\T)\otimes M_n\to C(\T)\otimes M_n$ denote the map $f\mapsto f\circ \lambda_{p}$. Then $\varphi(f) = \diag(\varphi_1(f), \varphi_2(f))$ for all $f\in C(\T)\otimes M_n$. For any $f\in C(\T,\u_n)$, $\varphi(f) = \widehat{\varphi}_1(f)\cdot \widehat{\varphi}_2(f)$, where $\widehat{\varphi}_{p} : C(\T,\u_n) \to C(\T,\u_{2n})$ are the maps
\[
\widehat{\varphi}_1(f)=\begin{pmatrix}
\varphi_1(f) & 0\\
0 & 1
\end{pmatrix}\text{ and } \widehat{\varphi}_2(f)=\begin{pmatrix}
1 & 0\\
0 & \varphi_2(f)
\end{pmatrix}.
\] 
Once again, since inclusion induces the identity map, we conclude that
\[
F_m(\varphi) = F_m(\varphi_1) + F_m(\varphi_2).
\]
Therefore, it remains to compute $F_m(\psi)$ when $\psi : C(\T)\otimes M_n\to C(\T)\otimes M_n$ is the $\ast$-homomorphism given by
\[
\psi(f) = f\circ \lambda,
\]
where $\lambda : \T\to \T$ is a continuous loop based at $1$. In that case, fix $m\geq 1$, and observe that $F_m$ is a homotopy invariant functor. Therefore, we may assume without loss of generality that $\lambda(z) = z^b$ where $b = w(\lambda)$ is the winding number of $\lambda$. Now let $C_{\ast}(\T,M_n) := \{f\in C(\T,M_n) : f(1) = 0\}$ and consider the split exact sequence
\[
\xymatrix{
0\ar[r] & C_{\ast}(\T,M_n) \ar[d]^{\widehat{\psi}} \ar[r] &C(\T,M_n)\ar[d]^{\psi}\ar[r]^{q} & M_n\ar[d]^{\text{id}}\ar[r] & 0 \\
0\ar[r] & C_{\ast}(\T,M_n) \ar[r] & C(\T,M_n)\ar[r]^{q} & M_n\ar[r] & 0
}
\]
where $q$ is the evaluation map at $1\in\T$ and $\widehat{\psi}$ is the restriction of $\psi$ to $C_{\ast}(\T,M_n)$. Under the isomorphism $F_m(C(\T)\otimes M_n) \cong V_n^m$, the map $F_m(\psi) : V_n^m\to V_n^m$ is given by
\[
F_m(\psi)(x,y) = (F_m(\text{id})(x), F_m(\widehat{\psi})(y)).
\]
Under the isomorphism $\pi_m(\u(C_{\ast}(\T, M_n))) \cong \pi_{m+1}(\u_n)$, $\widehat{\psi}$ induces the map $\pi_m(\widehat{\psi})([G]) = b[G]$. 	Thus, $F_m(\widehat{\psi})$ is realized as multiplication by $b$, proving the result.
\end{proof}

Now suppose $A = \bigoplus_{j=1}^K C(\T)\otimes M_{n_j}$ and $B = \bigoplus_{i=1}^L C(\T)\otimes M_{\ell_i}$, and suppose $\varphi : A\to B$ is a diagonal $\ast$-homomorphism. Then, each map
\[
\varphi_{i,j} : C(\T)\otimes M_{n_j} \hookrightarrow A \xrightarrow{\varphi} B \to C(\T)\otimes M_{\ell_i}
\]
is a diagonal $\ast$-homomorphism as in \cref{prop_fm_circle}. If $(a_{i,j}, b_{i,j})$ is the signature of $\varphi_{i,j}$, then we obtain an $L\times K$ matrix $\Phi \in M_{L\times K}(\N_0\times \N_0)$ given by
\[
\Phi := \begin{pmatrix}
(a_{1,1}, b_{1,1}) & (a_{1,2}, b_{1,2}) & \ldots & (a_{1,K}, b_{1,K}) \\
(a_{2,1}, b_{2,1}) & (a_{2,2}, b_{2,2}) & \ldots & (a_{2,K}, b_{2,K}) \\
\vdots & \vdots & \vdots & \vdots \\
(a_{L,1}, b_{L,1}) & (a_{L,2}, b_{L,2}) & \ldots & (a_{L,K}, b_{L,K})
\end{pmatrix}
\]
This is called the \emph{signature} matrix of $\varphi$. \cref{prop_iso_fm} and \cref{prop_fm_circle} together yield the following theorem.

\begin{theorem}\label{thm_homomorphism_circle}
Let $A = \bigoplus_{j=1}^K C(\T)\otimes M_{n_j}$ and $B = \bigoplus_{i=1}^L C(\T)\otimes M_{\ell_i}$, and suppose $\varphi : A\to B$ is a diagonal $\ast$-homomorphism with signature matrix $\Phi$ as above. Then, for any $m\geq 1$,
\[
F_m(\varphi) : \bigoplus_{j=1}^K V_{n_j}^m \to \bigoplus_{i=1}^L V_{\ell_i}^m
\]
is given by multiplication by $\Phi$, in the sense that
\[
F_m(\varphi)((x_j,y_j)_{1\leq j\leq K}) = \left(\sum_{j=1}^K a_{i,j}x_j, \sum_{j=1}^K b_{i,j}y_j\right)_{1\leq i\leq L}
\]
\end{theorem}

Finally, with the results of \cref{sec_hom_circle}, we arrive at the first main result of the paper. A few words of caution before we describe the result though. Given an A$\T$-algebra $A$, we may write $A$ as an inductive limit $A = \lim_{n\to\infty} (A_n,\psi_n)$ where each $A_n$ is a circle algebra, and $\psi_n : A_n\to A_{n+1}$ is a $\ast$-homomorphism. It is not, in general, possible to replace $\psi_n$ by a suitable diagonal $\ast$-homomorphism. However, when computing $F_m(\cdot)$ (or any other continuous homology theory), one may use the results of \cref{sec_hom_circle} to replace $\psi_n$ by maps that are diagonal. Now, the resulting inductive limit algebra is \emph{not necessarily} isomorphic to the algebra $A$. However, the functor $F_m(\cdot)$ is blind to the difference, which is what allows this theorem to work.

\begin{theorem}\label{thm_at_alg_rat_hom}
Let $A$ be an A$\T$-algebra that is expressed as an inductive limit of circle algebras
\[
A_1\xrightarrow{\psi_1} A_2\xrightarrow{\psi_2} A_3\to \ldots \to A.
\]
Then, for each $m\geq 1$, the group $F_m(A)$ may be computed as the inductive limit of a sequence
\[
F_m(A_1) \xrightarrow{\Phi_1} F_m(A_2)\xrightarrow{\Phi_2} F_m(A_3)\to \ldots \to F_m(A),
\]
where the connecting maps are given by multiplication by the signature matrices of certain diagonal $\ast$-homomorphism $\varphi_n : A_n\to A_{n+1}$.
\end{theorem}
\begin{proof}
Fix $m\geq 1$. Given an inductive sequence $A_1\xrightarrow{\psi_1} A_2 \xrightarrow{\psi_2} A_3 \to \ldots \to A$, we may assume by Thomsen's theorem (\cref{thomsen_1}) that each $\psi_i$ is a map of Type A. For each $i \in \N$, \cref{thm_homotopy_type_ab} tells us that there is a $\ast$-homomorphism $\eta_i : A_i \to A_{i+1}$ of Type B such that $\psi_i \sim_h \eta_i$. Suppose $B$ is the inductive limit of the sequence
\[
A_1\xrightarrow{\eta_1} A_2 \xrightarrow{\eta_2} A_3 \to \ldots \to B.
\]
Then, it follows that $F_m(A)$ is naturally isomorphic to $F_m(B)$ (even though $A$ and $B$ may not be isomorphic). Once again, by \cref{prop_type_bc}, we may replace $\eta_i$ by another map $\delta_i : A_i \to A_{i+1}$ that is of Type C and is unitarily equivalent to $\eta_i$, so that $B \cong \lim (A_n, \delta_n)$. Finally, by \cref{prop_type_cd}, there are diagonal $\ast$-homomorphism $\varphi_i : A_i\to A_{i+1}$ such that $\varphi_i \sim_h \delta_i$. If $A'$ denotes the inductive limit of the sequence
\[
A_1\xrightarrow{\varphi_1} A_2 \xrightarrow{\varphi_2} A_3 \to \ldots \to A',
\]
then $F_m(A) \cong F_m(A')$. The result now follows by appealing to \cref{thm_homomorphism_circle}.
\end{proof}

The next two examples illustrate this theorem.

\begin{ex}\label{ex_bunce_deddens}
Let $B$ be the Bunce-Deddens algebra given as the inductive limit of the sequence
\[
C(\T)\xrightarrow{\varphi_0} C(\T)\otimes M_2\xrightarrow{\varphi_1} C(\T)\otimes M_4\xrightarrow{\varphi_2} \to \ldots\to B
\]
where $\varphi_n(f):C(\T)\otimes M_{2^n}\to C(\T)\otimes M_{2^{n+1}}$ is given by 
\[
\varphi_n(f)(t) = (u(t)\otimes I_{2^n})\begin{pmatrix}
f(e^{\pi i t}) & 0\\
0 & f(e^{\pi i(1+t)})
\end{pmatrix}(u(t)\otimes I_{2^n})^{\ast}
\]
and $u:[0,1]\to \u_2$ is a path of unitaries in $M_2$ satisfying $u(0)=I$ and $u(1)=\begin{pmatrix}
0 & 1\\
1 & 0
\end{pmatrix}$. In order to calculate $F_m(B)$, we may replacing $\varphi_n$ by the recipe laid out in \cref{thm_homotopy_type_ab}, and consider the inductive sequence
\[
C(\T)\xrightarrow{\psi_0} C(\T)\otimes M_2\xrightarrow{\psi_1} C(\T)\otimes M_4\to \ldots
\]
where each $\psi_n$ is a diagonal $\ast$-homomorphism given by
\[
\psi_n(f)(z) =\begin{pmatrix}
f(1) & 0\\
0 & f(z)
\end{pmatrix}
\]
Here, the identity map on $\T$ has winding number $1$ and $\psi_n$ has multiplicity $2$. Hence for $m\geq 1$, $\psi_n$ has signature $(2,1)$. In other words, $F_m(\psi_n) : V_{2^n}^m \to V_{2^{n+1}}^m$ is given by
\[
F_m(\psi_n)(x,y)= (2x,y).
\]
Now,
\[
F_m(C(\T)\otimes M_{2^n}) = \begin{cases}
\Q\oplus 0 &: \text{ if } 1 \leq m\leq 2^{n+1}-1, m \text{ odd}\\
0\oplus \Q &: \text{ if } 1\leq m\leq 2^{n+1}-1, m \text{ even}\\
0 &: \text{ otherwise}.
\end{cases}
\]
Therefore, $F_m(B) \cong \Q$ for all $m \in \N$.
\end{ex}

\begin{ex}\label{ex_goodearl}
Let $(r_n)$ and $(p_n)$ be sequences of positive integers such that $r_n$ divides $r_{n+1}$ and $p_n< r_{n+1}/r_n$ for each $n \in \N$. Choose a finite set $F_n = \{z_{n,1}, z_{n,2}, \ldots, z_{n,p_n}\} \subset \T$, and define $\psi_n : C(\T)\otimes M_{r_n} \to C(\T)\otimes M_{r_{n+1}}$ by
\[
\psi_n(f)(z)=\diag(f(z_{n,1}), f(z_{n,2}),\ldots,f(z_{n,p_n}), f(z),f(z),\ldots, f(z)).
\]
Consider the algebra $G$ given as an inductive limit of the sequence
\[
C(\T)\otimes M_{r_1}\xrightarrow{\psi_1} C(\T)\otimes M_{r_2}\xrightarrow{\psi_2} C(\T)\otimes M_{r_3}\to \ldots\to G.
\]
If $\varphi_n : C(\T)\otimes M_{r_n} \to C(\T)\otimes M_{r_{n+1}}$ denotes the map
\[
\varphi_n(f)(z) = \diag(f(1), f(1), \ldots, f(1), f(z),f(z),\ldots, f(z)),
\]
then $\psi_n\sim_h \varphi_n$. Observe that for $m\geq 1$, the signature of $\varphi_n$ is $(r_{n+1}/r_n, r_{n+1}/r_n-p_n)$. Therefore, for each $m\geq 1$, $F_m(\psi_n) : V_{r_n}^m \to V_{r_{n+1}}^m$ is thus given by 
\[
F_m(\psi_n)(x,y)= ((r_{n+1}/r_n)x,(r_{n+1}/r_n-p_n)y)).
\]
Hence, $F_m(G)=\Q$ for all $m\geq 1$ (since $r_{n+1}/r_n > p_n \geq 1$). Note that if $\bigcup_{n=k}^{\infty} F_n$ is dense in $\T$ for each $k \in \N$, then $G$ is simple and thus a Goodearl algebra.
\end{ex}

\section{K-Stability}\label{sec_k_stability_at}

We now turn to a proof of \cref{mainthm_k_stable} and begin with the definition of K-stability. As mentioned before, this was first studied by Thomsen \cite{thomsen}. Its rational counterpart was also discussed by Farjoun and Schochet \cite{farjoun}.

\begin{defn}\label{defn:k_stable}
Let $A$ be a C*-algebra and $j\geq 2$. Define $\iota^A_j: M_{j-1}(A)\to M_j(A)$ to be the natural inclusion map
\[
a\mapsto\begin{pmatrix}
a&0\\
0&0
\end{pmatrix}.
\]
$A$ is said to be \emph{$K$-stable} if $G_k(\iota^A_j): G_k(M_{j-1}(A))\to G_k(M_j(A))$ is an isomorphism for all $k\geq 0$ and all $j\geq 2$. $A$ is said to be \emph{rationally $K$-stable} if $F_m(\iota^A_j):F_m(M_{j-1}(A))\to F_m(M_j(A))$ is an isomorphism for all $m\geq 1$ and all $j\geq 2$.
\end{defn}

Note that, for a $K$-stable C*-algebra, $G_k(A) \cong K_{k+1}(A)$ and for a rationally $K$-stable $C^{\ast}$-algebra, $F_m(A) \cong K_{m+1}(A)\otimes \Q$. A variety of interesting C*-algebras are known to be $K$-stable (see \cite[Remark 1.5]{apurva_pv_cx}). Clearly, $K$-stability implies rational $K$-stability. By \cite[Theorem B]{apurva_pv_af}, the converse is true for AF-algebras. However, the converse is not true in general (see \cite[Example 2.1]{apurva_pv_cx_rat}). The aim of this section is to show that, for the class of A$\T$-algebras, both these notions are equivalent. We begin with the following observation.

\begin{rem}\label{rem_generic_hom_limit}
Let $A$ be an A$\T$-algebra, given as an inductive limit $A = \lim (A_n, \varphi_n)$. By \cref{thomsen_1}, we may assume that each $\varphi_n$ is of Type A. In \cref{thm_homotopy_type_ab}, we proved that there is another inductive sequence $(A_n, \psi_n)$, where each $\psi_n$ is a $\ast$-homomorphism of Type B such that $\varphi_n \sim_h \psi_n$ for each $n\in \N$. Hence, the inductive limit $A' = \lim (A_n, \psi_n)$ has the property that $G_k(A) \cong G_k(A')$ for each $k\geq 0$. \\

Now suppose $A$ is $K$-stable, then we prove that $A'$ is also $K$-stable (the same argument applies for rational $K$-stability as well). To do this, write $\alpha_n : A_n\to A$ and $\beta_n : A_n\to A'$ be the $\ast$-homomorphisms defining $A$ and $A'$ respectively. In other words, $\alpha_{n+1}\circ \varphi_n = \alpha_n$ and $\beta_{n+1}\circ \psi_n = \beta_n$ for all $n\in \N$. Fix $k\geq 0$. For simplicity, we will show that $\eta := G_k(\iota^{A'}_2) : G_k(A')\to G_k(M_2(A'))$ is an isomorphism, under the assumption that $\rho := G_k(\iota^A_2) : G_k(A) \to G_k(M_2(A))$ is an isomorphism. Now observe that the following diagrams commute
\[
\xymatrix{
& G_k(A_i)\ar[rr]^{(\iota^{A_i}_2)_{\ast}}\ar[ld]_{(\alpha_i)_{\ast}}\ar[dd]_{(\varphi_i)_{\ast}} & & G_k(M_2(A_i))\ar[ld]_{(\alpha^{(2)}_i)_{\ast}}\ar[dd]_{(\varphi^{(2)}_i)_{\ast}} \\
G_k(A)\ar@{..>}[rr]^{\rho} && G_k(M_2(A)) \\
& G_k(A_{i+1})\ar[rr]_{(\iota^{A_{i+1}}_2)_{\ast}}\ar[ul]^{{(\alpha_{i+1})_{\ast}}} & & G_k(M_2(A_{i+1}))\ar[ul]^{{(\alpha^{(2)}_{i+1})_{\ast}}} \\
& G_k(A_i)\ar[rr]^{(\iota^{A_i}_2)_{\ast}}\ar[ld]_{(\beta_i)_{\ast}}\ar[dd]_{(\psi_i)_{\ast}} & & G_k(M_2(A_i))\ar[ld]_{(\beta^{(2)}_i)_{\ast}}\ar[dd]_{(\psi^{(2)}_i)_{\ast}} \\
G_k(A')\ar@{..>}[rr]^{\eta} && G_k(M_2(A')) \\
& G_k(A_{i+1})\ar[rr]_{(\iota^{A_{i+1}}_2)_{\ast}}\ar[ul]^{{(\beta_{i+1})_{\ast}}} & & G_k(M_2(A_{i+1}))\ar[ul]^{{(\beta^{(2)}_{i+1})_{\ast}}}
}
\]
(Note that if $\theta : C\to D$ is a $\ast$-homomorphism, we write $\theta^{(2)} : M_2(C)\to M_2(D)$ for the induced $\ast$-homomorphism). Since $(\varphi_i)_{\ast} = (\psi_i)_{\ast}$ and $(\varphi_i^{(2)})_{\ast} = (\psi_i^{(2)})_{\ast}$, the universal property of the inductive limit also tells us that there are isomorphisms $\lambda : G_k(A)\to G_k(A')$ and $\mu : G_k(M_2(A))\to G_k(M_2(A'))$ such that the following diagrams commute
\[
\xymatrix{
& G_k(A_i)\ar[dl]_{(\alpha_i)_{\ast}}\ar[rd]^{(\beta_i)_{\ast}} & & & G_k(M_2(A_i))\ar[dl]_{(\alpha^{(2)}_i)_{\ast}}\ar[rd]^{(\beta^{(2)}_i)_{\ast}} \\
G_k(A)\ar[rr]^{\lambda} & & G_k(A') & G_k(M_2(A))\ar[rr]^{\mu} && G_k(M_2(A'))
}
\]
A short argument shows that
\[
\eta \circ \lambda \circ (\alpha_i)_{\ast} = (\beta^{(2)}_i)_{\ast}\circ (\iota^{A_i}_2)_{\ast} = \mu \circ \rho\circ (\alpha_i)_{\ast}.
\]
Hence, $\eta\circ \lambda = \mu\circ \rho$. In this expression, both $\lambda$ and $\mu$ are isomorphisms. Therefore, if $\rho$ is an isomorphism, so is $\eta$. Now, the same argument applies for each such map $G_k(M_{n-1}(A'))\to G_k(M_n(A'))$ and we conclude that $A'$ is $K$-stable. Since the relationship between $A$ and $A'$ is symmetric, we see that $A$ is $K$-stable if and only if $A'$ is $K$-stable.\\

Therefore, when considering questions of $K$-stability (or rational $K$-stability), we need only consider the case where the connection homomorphisms in the inductive limit are of Type B. In fact, we may repeat this argument using \cref{prop_type_bc} and \cref{prop_type_cd} and assume that the connecting maps are diagonal. This is what we now do.
\end{rem}

Suppose $A$ is an A$\T$-algebra, and $(A_n,\varphi_n)$ is an inductive sequence such that $A = \lim (A_n,\varphi_n)$, and each $\varphi_n$ is a diagonal $\ast$-homomorphism. For each $n \in \N$, let $B_n$ be the quotient of $A_n$ obtained by evaluating each summand of $A_n$ at $1\in \T$. Since the connecting maps are diagonal, we get an inductive sequence $(B_n, \psi_n)$ of finite dimensional C*-algebras and maps $\pi_n : A_n\to B_n$ such that $\psi_n\circ \pi_n = \pi_{n+1}\circ \varphi_n$ for all $n\in \N$. If $B = \lim (B_n, \psi_n)$, then

\begin{lem}\label{lem_rat_k_stable_af}
If $A$ is rationally K-stable, so is $B$.
\end{lem}
\begin{proof}
It suffices to show that the map $F_m(\iota^B_2) : F_m(B)\to F_m(M_2(B))$ is an isomorphism for each $m\geq 1$. If $m$ is even, $F_m(B) = F_m(M_2(B)) = 0$ by \cite[Lemma 3.2]{apurva_pv_af}, so let $m \in \N$ be odd. For each $n\in \N$, consider the short exact sequence
\[
0 \to C_{\ast}(\T, B_n) \to A_n \xrightarrow{\pi_n} B_n \to 0
\]
Since the connecting maps $\varphi_n : A_n\to A_{n+1}$ are diagonal, the restriction gives a $\ast$-homomorphism $\widetilde{\varphi_n} : C_{\ast}(\T, B_n) \to C_{\ast}(\T, B_{n+1})$, and we have a commuting diagram of extensions
\[
\xymatrix{
0\ar[r] & C_{\ast}(\T, B_n)\ar[r]\ar[d]_{\widetilde{\varphi_n}} & A_n\ar[r]^{\pi_n}\ar[d]^{\varphi_n} & B_n\ar[r]\ar[d]^{\psi_n} & 0 \\
0\ar[r] & C_{\ast}(\T, B_{n+1})\ar[r] & A_{n+1}\ar[r]^{\pi_{n+1}} & B_{n+1}\ar[r] & 0 \\
}
\]
Now, $F_m(C_{\ast}(\T, B_j)) \cong F_{m+1}(B_j) = 0$ for all $j\in \N$ by \cite[Lemma 3.2]{apurva_pv_af}, and thus $F_m(\pi_n) : F_m(A_n) \to F_m(B_n)$ is an isomorphism. This induces an isomorphism $\theta : F_m(A) \to F_m(B)$. \\

Furthermore, the same argument applies with $M_2(A_n)$ instead of $A_n$, and we obtain an isomorphism $F_m(\pi_n^{(2)}) : F_m(M_2(A_n))\to F_m(M_2(B_n))$ such that the following diagram commutes
\[
\xymatrixcolsep{3pc}\xymatrix{
F_m(A_n)\ar[r]^{F_m(\pi_n)}\ar[d]_{F_m(\iota^{A_n}_2)} & F_m(B_n)\ar[d]^{F_m(\iota^{B_n}_2)} \\
F_m(M_2(A_n))\ar[r]^{F_m(\pi_n^{(2)})} & F_m(M_2(B_n))
}
\]
Once again, there is an induced isomorphism $\rho : F_m(M_2(A))\to F_m(M_2(B))$. Since each horizontal map is an isomorphism, we obtain a commuting diagram at the level of inductive limits
\[
\xymatrix{
F_m(A)\ar[r]^{\theta}\ar[d]_{F_m(\iota^{A}_2)} & F_m(B)\ar[d]^{F_m(\iota^{B}_2)} \\
F_m(M_2(A))\ar[r]^{\rho} & F_m(M_2(B))
}
\]
Since $F_m(\iota^A_2)$ is an isomorphism, so is $F_m(\iota^B_2)$.
\end{proof}

We now need to revisit some notation from \cite[Section 3]{apurva_pv_af} and adapt it to our setting here. If $C$ is a finite dimensional C*-algebra, we write
\[
\min\dim(C) = \min\{\text{square root of the dimension of a simple summand of } C\}.
\]
In other words, if $C = \bigoplus_{j=1}^K M_{n_j}$, then $\min\dim(C) = \min\{n_j : 1\leq j\leq K\}$. Also, we write $C^{(j)}$ to be the direct sum of all simple summands of $C$ of dimension equal to $j^2$ (adopting the convention that the direct sum over an empty index set is the zero C*-algebra). Similarly, $C^{(>j)}$ denotes the direct sum of all simple summands of $C$ whose dimension is $>j^2$, and $C^{(<j)}$ is the direct sum of all simple summands of $C$ whose dimension is $<j^2$.  Hence,
\[
C = C^{(<j)}\oplus C^{(j)}\oplus C^{(>j)}.
\] 
Now if $A = C(\T)\otimes C$ is a circle algebra and $C$ is finite dimensional, we write $\min\dim(A) = \min\dim(C)$, $A^{(j)} := C(\T)\otimes C^{(j)}$, and $A^{(>j)}$ and $A^{(<j)}$ are also defined analogously. \\

We are now ready to prove \cref{mainthm_k_stable}, and we do so using the following lemmas. The first of these lemmas is an analogue (and indeed a consequence) of \cite[Lemma 3.7]{apurva_pv_af}.

\begin{lem}\label{lem_rat_k_stable_min_dim}
Let $A$ be a rationally $K$-stable A$\T$-algebra, and suppose $A = \lim (A_p,\varphi_p)$ where each $A_p$ is a circle algebra and each $\varphi_p$ is a diagonal $\ast$-homomorphism. Then, for each $m \in \N$, there is a sequence $(A_{m,p}, \varphi_p^{m})$ of circle algebras with diagonal connecting maps such that $A = \lim_{p\to \infty} (A_{m,p}, \varphi_p^{m})$ and
\[
\min\dim(A_{m,p}) \geq m
\]
for all $p \in \N$.
\end{lem}
\begin{proof}
The result is true if $m=1$, so we may assume that the result is true for $0 \leq i\leq m$, and now construct the sequence $(A_{m+1,p}, \varphi_p^{m+1})$. Let $(A_{m,p}, \varphi_p^m)$ be an inductive sequence such that $A = \lim_{p\to \infty} (A_{m,p}, \varphi_p^m)$ and $\min\dim(A_{m,p}) \geq m$ for all $p \in \N$. Taking quotients as before, we let $(B_{m,p}, \psi_p^{m})$ be the associated sequence of finite dimensional algebras obtained from $(A_{m,p},\varphi_p^{m})$, and let $\pi_p^{m} : A_{m,p} \to B_{m,p}$ to be the natural quotient maps (evaluation at $1\in \T$). If $B_m := \lim (B_{m,p}, \psi_p^{m})$, then $B_m$ is rationally $K$-stable by \cref{lem_rat_k_stable_af}. \\

We now adopt the approach of \cite[Lemma 3.7]{apurva_pv_af}. By hypothesis, $\min\dim(A_{m,p})\geq m$ for each $p \in \N$. If $\min\dim(A_{m,p})\geq m+1$ for all but finitely many $p$, then we may ignore these finitely many terms and write $A_{m+1,p} = A_{m,p}$ and $\psi_p^{m+1} = \psi_p^{m}$ for all $n\in \N$. Therefore, we may assume that $\min\dim(A_{m,p}) = m$ for infinitely many $p\in \N$, and also that $\min\dim(A_{m,1}) = m$. \\

Since the connecting maps $\varphi_p^{m} : A_{m,p}\to A_{m,p+1}$ are diagonal maps, it follows that the connecting maps $\psi_p^{m} : B_{m,p} \to B_{m,p+1}$ are injective. By the argument of \cite[Lemma 3.7]{apurva_pv_af}, there is a subsequence $(B_{m,n_j}, \psi_{n_j}^{m})$ such that $B_{m,n_j}^{(m)}$ is an orphan for each $j \in \N$ (in the sense that $B_{m,n_j}^{(m)}$ is not the target of any arrow emanating from $B_{m,n_j-1}$). In the sequence $(A_{m,p}, \varphi^{m}_p)$, this means that $A_{m,n_j}^{(m)}$ must also be an orphan. Let $\iota_j : A_{m,n_j}^{(>m)} \to A_{m,n_j}$ and $\pi_j : A_{m,n_j}\to A_{m,n_j}^{(>m)}$ be the natural inclusion and quotient maps respectively. Then, it follows that
\[
\iota_j \circ \pi_j \circ \varphi_{n_j,n_{j-1}}^{m} = \varphi_{n_j,n_{j-1}}^{m}
\]
for all $j\geq 1$ with the convention that $n_0 = 1$. Hence, the following diagram commutes
\[
\xymatrixcolsep{4pc}\xymatrix{
\ldots\ar[r] & A_{m,n_{j-1}}\ar[rr]^{\varphi^{m}_{n_j,n_{j-1}}}\ar[d]_{\pi_j\circ\varphi^{m}_{n_j,n_{j-1}}} && A_{m,n_j}\ar[d]^{\pi_{j+1}\circ\varphi^{m}_{n_{j+1},n_j}}\ar[r]^{\varphi^{m}_{n_{j+1},n_j}} & \ldots \\
\ldots\ar[r] & A_{m,n_j}^{(>m)}\ar[rr]_{\pi_{j+1}\circ\varphi^{m}_{n_{j+1},n_j}\circ \iota_j}\ar[rru]^{\iota_j} && A_{m,n_{j+1}}^{(>m)}\ar[r] & \ldots
}
\]
We set $(A_{m+1,j}, \varphi^{m+1}_j)$ to be the terms in the lower row. Then, it follows from \cite[Exercise 6.8]{rordam}, that $\lim (A_{m+1,j}, \varphi^{m+1}_j) \cong A$. Furthermore, by construction, we have $\min\dim(A_{m+1,j})\geq m+1$ for all $j\in \N$. Finally, since each $\varphi_j^{m}$ is a diagonal $\ast$-homomorphism, so is $\varphi_j^{m+1}$.
\end{proof}

We are now in a position to prove the first part of \cref{mainthm_k_stable}. The argument follows the same line of reasoning as that of \cite[Theorem 3.8]{apurva_pv_af}.

\begin{theorem}\label{thm_at_alg_k_stable}
An A$\T$-algebra $A$ is $K$-stable if and only if it is rationally $K$-stable.
\end{theorem}
\begin{proof}
It suffices to prove that rational $K$-stability implies $K$-stability. So suppose $A$ is a rationally $K$-stable A$\T$-algebra. By \cref{rem_generic_hom_limit}, we may assume that $A$ is described as an inductive limit $A = \lim (A_p, \varphi_p)$, where each connecting map $\varphi_p : A_p\to A_{p+1}$ is a diagonal $\ast$-homomorphism. \\

Now fix $k \in \N$, and we wish to prove that $G_k(\iota^A_j) : G_k(M_{j-1}(A)) \to G_k(M_j(A))$ is an isomorphism for each $j \geq 2$. Since $A$ is rationally $K$-stable, we may assume by \cref{lem_rat_k_stable_min_dim} that
\[
\min\dim(A_p) \geq \left\lceil \frac{k}{2}\right\rceil + 1.
\]
for each $p \in \N$. Now, $M_{j-1}(A)$ is an inductive limit of the algebras $\{M_{j-1}(A_p)\}$ and
\[
\min\dim(M_{j-1}(A_p)) \geq \left\lceil \frac{k}{2}\right\rceil + 1.
\]
for each $p\in \N$. Therefore, replacing $M_{j-1}(A)$ by $A$, it suffices to prove that the inclusion map $\iota_2^A : A\to M_2(A)$ induces an isomorphism $G_k(\iota_2^A) : G_k(A)\to G_k(M_2(A))$. \\

Now write $A_p = C(\T)\otimes B_p$ where $B_p = M_{\ell_1^p}(\C)\oplus M_{\ell_2^p}(\C) \oplus \ldots \oplus M_{\ell_{k_p}^p}(\C)$. By construction,
\[
k+1 < 2\ell_i^p
\]
for all $p$ and all $1\leq i\leq k_p$. Hence, the inclusion maps $\iota_2^{B_p} : B_p\to M_2(B_p)$ induce isomorphisms $G_k(B_p) \to G_k(M_2(B_p))$ for all $p \in \N$ by \cite[Chapter 2, Corollary 3.17]{mimura}. Also, by the natural isomorphisms $G_k(C_{\ast}(\T, B_p)) \cong G_{k+1}(B_p)$ and $G_k(M_2(C_{\ast}(\T, B_p))) \cong G_k(C_{\ast}(\T, M_2(B_p))) \cong G_{k+1}(M_2(B_p))$, the map
\[
G_k(\iota_2^{C_{\ast}(\T,B_p)}) : G_k(C_{\ast}(\T, B_p)) \to G_k(M_2(C_{\ast}(\T, B_p))
\]
is also an isomorphism. Furthermore, the following diagram commutes
\[
\xymatrix{
0\ar[r] & C_{\ast}(\T, B_p)\ar[r]\ar[d]_{\iota_2^{C_{\ast}(\T,B_p)}} & A_p\ar[r]^{\pi_p}\ar[d]^{\iota_2^{A_p}} & B_p\ar[r]\ar[d]^{\iota_2^{B_p}} & 0 \\
0\ar[r] & M_2(C_{\ast}(\T, B_p))\ar[r] & M_2(A_p)\ar[r]^{\pi_p^{(2)}} & M_2(B_p)\ar[r] & 0 \\
}
\]
Since $(G_k)$ is a homology theory, this induces a diagram of long exact sequences. By the Five Lemma, $G_k(\iota^{A_p}_2) : G_k(A_p) \to G_k(M_2(A_p))$ is an isomorphism for each $p \in \N$. Once again, the following diagram commutes
\[
\xymatrix{
0\ar[r] & C_{\ast}(\T, B_p)\ar[r]\ar[d]_{\widetilde{\varphi_p}} & A_p\ar[r]^{\pi_p}\ar[d]^{\varphi_p} & B_p\ar[r]\ar[d]^{\psi_p} & 0 \\
0\ar[r] & C_{\ast}(\T, B_{p+1})\ar[r] & A_{p+1}\ar[r]^{\pi_{p+1}} & B_{p+1}\ar[r] & 0 \\
}
\]
Therefore, $G_k(\iota_2^A) : G_k(A) \to G_k(M_2(A))$ is an isomorphism, and $A$ is $K$-stable.
\end{proof}

Our proof of the next part of \cref{mainthm_k_stable} is now a refinement of these arguments given above. We now recall an important definition (see, for instance, \cite[Definition 3.1.1]{rordam_stormer}).

\begin{defn}\label{defn_slow_dim_growth}
Let $(A_p,\varphi_p)$ be a sequence of circle algebras, then the sequence is said to have \emph{slow dimension growth} if $\lim_{p\to\infty} \min\dim(A_p) = +\infty$. An A$\T$-algebra $A$ is said to have slow dimension growth if it is the inductive limit of some sequence that has slow dimension growth.
\end{defn}

\begin{theorem}\label{thm_k_stable_sdg}
An A$\T$-algebra is $K$-stable if and only if it has slow dimension growth.
\end{theorem}
\begin{proof}
Let $A$ be an A$\T$-algebra. If $A$ has slow dimension growth, then for each $m \in \N$, there is a sequence $(A_{m,p}, \varphi^{m}_p)$ of circle algebras such that $A = \lim_{p\to \infty} (A_{m,p}, \varphi^{m}_p)$ and $\min\dim(A_{m,p}) \geq m$ for each $p \in \N$. By the results of \cref{sec_hom_circle}, we may assume that $\varphi^{m}_p$ are all diagonal, therefore proof of \cref{thm_at_alg_k_stable} applies and  shows that $A$ must be $K$-stable. \\

Conversely, suppose $A$ is $K$-stable, then $A$ is rationally $K$-stable. Write $A = \lim_{p \to \infty} (A_{1,p}, \varphi^{1}_p)$ where each $A_{1,p}$ is a circle algebra and each $\varphi^{1}_p$ is of type A . By \cref{rem_generic_hom_limit}, there is a sequence $(\widetilde{A}_{1,p}, \psi^{1}_p)$ such that $\widetilde{A}_{1,p} = A_{1,p}$ for all $p \in \N$, whose connecting maps are diagonal maps and whose limit $\widetilde{A} = \lim_{p\to \infty} (\widetilde{A}_{1,p},\psi^{1}_p)$ is also rationally K-stable.  \\

By the proof of \cref{lem_rat_k_stable_min_dim}, there is a subsequence $(p(1,j))_{j=1}^{\infty} \subset \N$ such that $\widetilde{A}_{1,p(1,j)}^{(1)}$ is an orphan for each $j \in \N$. By the construction in \cref{thm_homotopy_type_ab}, it is then clear that $A_{1,p(1,j)}^{(1)}$ must also be an orphan in the original sequence $(A_{1,p}, \varphi_p^1)$ for each $j \in \N$. Let $A_{2,j} = A_{1,p(1,j)}^{(>1)}$. Then, as in \cref{lem_rat_k_stable_min_dim}, we get an inductive sequence $(A_{2,p}, \varphi^2_p)$ such that $A = \lim_{p\to \infty} (A_{2,p}, \varphi^2_p)$ with $\min\dim(A_{2,p}) \geq 2$ for each $p \in \N$. \\

Thus proceeding, for each $m \in \N$, there is a sequence $(A_{m,p}, \varphi_p^m)$ of circle algebras such that $A = \lim_{p\to \infty}(A_{m,p}, \varphi_p^m)$ and $\min\dim(A_{m,p}) \geq m$ for each $p \in \N$. Moreover, for each $m \in \N$, there is a subsequence $(p(m,j))_{j=1}^{\infty} \subset \N$ and quotient maps $\pi^m_j : A_{m,p(m,j)} \to A_{m+1,j}$ such that the following diagram commutes
\[
\xymatrixcolsep{5pc}\xymatrix{
A_{m,p(m,j)}\ar[d]^{\pi^m_j}\ar[r]^{\varphi^{m}_{p(m,j+1),p(m,j)}} & A_{m,p(m,j+1)}\ar[d]^{\pi^m_{j+1}} \\
A_{m+1,j}\ar[r]_{\varphi^{m+1}_j} & A_{m+1,j+1}
}
\]
where $\varphi^{m}_{k,\ell} = \varphi_{k-1}\circ \varphi_{k-2}\circ \ldots \circ \varphi_{\ell} : A_{m,\ell} \to A_{m,k}$ whenever $k>\ell$. Furthermore, we may assume without loss of generality that $p(m,j+1)\geq j$ for each $m,j\in \N$. Now define $\psi_m : A_{m,1} \to A_{m+1,1}$ by
\[
\psi_m := \pi^m_1\circ \varphi^m_{p(m,1),1}.
\]
Then, $(A_{m,1},\psi_m)$ is an inductive sequence with $\min\dim(A_{m,1})\geq m$ for each $m \in \N$. We claim that $A = \lim_{m\to \infty} (A_{m,1},\psi_m)$. To see this, let $B := \lim_{m\to \infty} (A_{m,1},\psi_m)$, and let $\beta_m : A_{m,1}\to B$ be $\ast$-homomorphisms defining $B$ such that $\beta_{m+1}\circ \psi_m = \beta_m$. Also, let $\alpha^m_p : A_{m,p} \to A$ be the $\ast$-homomorphisms such that $\alpha^m_{p+1}\circ \varphi^m_p = \alpha^m_p$ for each $m,p \in \N$. Then, it follows that the following diagram commutes
\[
\xymatrix{
A_{m,1}\ar[rd]_{\alpha^m_1}\ar[rr]^{\psi_m} && A_{m+1,1}\ar[ld]^{\alpha^{m+1}_1} \\
& A
}
\]
By the universal property of the inductive limit, there is a $\ast$-homomorphism $\lambda : B\to A$ such that $\lambda\circ \beta_m = \alpha^m_1$ for each $m \in \N$. We claim that $\lambda$ is an isomorphism. \\

For surjectivity, it suffices to show that
\[
\bigcup_{k=1}^{\infty} \alpha^1_k(A_{1,k}) = \bigcup_{m=1}^{\infty} \alpha^m_1(A_{m,1})
\]
since $A = \lim_{k\to \infty} (A_{1,k}, \varphi^1_k)$. So fix $a \in \bigcup_{k=1}^{\infty} \alpha^1_k(A_{1,k})$, and choose $k_1\in \N$ such that $a \in \alpha^1_{k_1}(A_{1,k_1})$ and assume $k_1 > 1$. Since $p(1,k_1-1) \geq k_1$ by hypothesis, it follows that
\[
a \in \alpha^1_{p(1,k_1-1)}(A_{1,p(1,k_1-1)})
\]
However, for each $m,j\in \N$, the following diagram commutes
\[
\xymatrix{
A_{m,p(m,j)}\ar[rr]^{\pi^m_j}\ar[rd]_{\alpha^m_{p(m,j)}} && A_{m+1,j}\ar[ld]^{\alpha^{m+1}_j} \\
& A
}
\]
Therefore, $a\in \alpha^2_{k_2}(A_{2,k_2})$ where $k_2 = k_1-1 < k_1$. If $k_2 = 1$, then we may stop this process. Else, we may repeat this until we obtain $N \in \N$ such that $a \in \alpha^N_1(A_{N,1})$. Thus, $a \in \bigcup_{m=1}^{\infty} \alpha^m_1(A_{m,1})$ and we have proved that $\lambda$ is surjective. \\

For injectivity of $\lambda$, it suffices to prove that $\ker(\alpha^m_1) \subset \ker(\beta_m)$ for each $m \in \N$ (by \cite[Proposition 6.2.4]{rordam}). So fix $m \in \N$ and let $a \in \ker(\alpha^m_1)$. Then, $\lim_{n\to \infty}\|\varphi^m_{n,1}(a)\| = 0$. Therefore, if $\epsilon > 0$, then there exists $k_1 \in \N$ such that $\|\varphi^m_{k_1,1}(a)\| < \epsilon$. Assume first that $k_1 > 1$. Once again, since $p(m,k_1-1) \geq k_1$, it follows that
\[
\|\varphi^m_{p(m,k_1-1),1}(a)\| < \epsilon,
\]
and thus $\|\pi^m_{k_1-1}\circ \varphi^m_{p(m,k_1-1),1}(a)\| < \epsilon$. However, the following diagram commutes
\[
\xymatrixcolsep{5pc}\xymatrix{
A_{m,1}\ar[d]^{\psi_m}\ar[r]^{\varphi^m_{p(m,k_1-1),1}} & A_{m,p(m,k_1-1)}\ar[d]^{\pi^m_{k_1-1}} \\
A_{m+1,1}\ar[r]^{\varphi^{m+1}_{k_1-1}} & A_{m+1,k_1-1}
}
\]
Therefore, $\|\varphi^{m+1}_{k_2}\circ \psi_m(a)\| < \epsilon$ where $k_2 = k_1-1 < k_1$. Thus proceeding, there exists $N\in \N$ such that $\|\varphi^N_1\circ \psi_{N,m}(a)\| < \epsilon$ (where $\psi_{N,m} = \psi_{N-1}\circ \psi_{N-2}\circ \ldots \circ \psi_m$). This implies that $\|\varphi^N_{p(N,1),1}\circ \psi_{N,m}(a)\| < \epsilon$, which in turn proves that
\[
\|\psi_{N+1,m}(a)\| < \epsilon.
\]
This is true for any $\epsilon > 0$, so $\lim_{n\to \infty} \|\psi_{n,m}(a)\| = 0$. In other words, $a\in \ker(\beta_m)$. We conclude that $\ker(\alpha^m_1)\subset \ker(\beta_m)$, so $\lambda$ is an isomorphism. \\

Thus, $A = \lim_{m\to \infty} (A_{m,1},\psi_m)$. Since $\min\dim(A_{m,1}) \geq m$ for each $m\in \N$, we conclude that $A$ has slow dimension growth.
\end{proof}

\begin{ex}
\cref{mainthm_k_stable} has a number of interesting consequences, some of which we describe below.
\begin{enumerate}
\item If $A$ is a simple, infinite dimensional A$\T$-algebra, then $A$ has slow dimension growth (see \cite{dnnp}). Therefore, every such algebra is $K$-stable.
\item In particular, the simple noncommutative tori are all $K$-stable. However, this class of C*-algebras is known to be $K$-stable from the results of Rieffel \cite{rieffel2}.
\item Similarly, the Bunce-Deddens algebra from \cref{ex_bunce_deddens} is $K$-stable because it is simple \cite{bunce_deddens}. Once again, this may be deduced from another result due to Zhang which states that any non-elementary, simple C*-algebra with real rank zero and stable rank one is $K$-stable \cite{zhang}.
\item Finally, the algebra $G$ from \cref{ex_goodearl} is also $K$-stable because it has slow dimension growth. Notice that $G$ need not be simple in general.
\end{enumerate}
\end{ex}

\textbf{Acknowedgements:} The first named author is supported by UGC Junior Research Fellowship No. 1229, and the second named author was supported by the SERB Grant No. MTR/2020/000385.


\end{document}